\documentclass{amsart}
\usepackage{graphicx}
\usepackage[mathscr]{eucal}
\usepackage{amssymb}
\newtheorem{thm}{Theorem}

\newtheorem{lem}[thm]{Lemma}
\newtheorem{prop}[thm]{Proposition}
\theoremstyle{definition}

\theoremstyle{remark}

\theoremstyle{example}

\newtheorem{qu}[thm]{Question}

\newcommand{\set}[1]{\left\{#1\right\}}

\newcommand{\A}{\mathcal{A}}

\newcommand{\Pm}{\mathrm{Prim}}
\newcommand{\mset}{\emptyset}

\newcommand{\st}{$C^*$}
\newcommand{\pr}{\prime}
\newcommand{\BH}{$\mathcal{B}(H)$}
\newcommand{\Int}{\mathrm{Int}}
\begin{document}
\baselineskip=18pt
\title{Continuous Fields of Postliminal $C^*$-algebras}
\author{Aldo J. Lazar}
\address{School of Mathematical Sciences\\
         Tel Aviv University\\
         Tel Aviv 69978, Israel}
\email{aldo@post.tau.ac.il}

\thanks{We wish to express our gratitude to Professor E. Christensen and Professor J. Phillips for providing us the reference \cite{C2}.}%
\subjclass{46L05} \keywords{continuous fields of $C^*$-algebras, postliminal $C^*$-algebras, liminal $C^*$ algebras}
\date{April 22, 2014}

\begin{abstract}

   We discuss a problem of Dixmier \cite[Problem 10.10.11]{D} on continuous fields of postliminal \st -algebras and the greatest liminal ideals of the
   fibers.

\end{abstract}
\maketitle

\section{Introduction}

In \cite[Problem 10.10.11]{D} Dixmier asked the following question: given a continuous field $((A(t)),\Theta)$ of postliminal \st -algebras over
some topological space $T$ and with $B(t)$ the greatest liminal ideal of $A(t)$ and $\Theta^{\pr} := \{x\in \Theta \mid x(t)\in B(t), t\in T\}$
is $((B(t)),\Theta^{\pr})$ a continuous field of \st -algebras? We shall call a continuous field of postliminal \t -algebras for which the
answer to this question is affirmative a tame continuous field.

An example of a continuous field that is not tame can be constructed over $T := \mathbb{N}\cup \{\infty\}$. We let $A(n)$, $n\in \mathbb{N}$, be
the unitization of $K(H)$, the algebra of all compact operators over an infinite dimensional Hilbert space $H$, and $A(\infty) :=
\mathbb{C}I_H$, $I_H$ being the identity operator on $H$. $\Theta$ consists of all the fields $x$ such that $x(n) = \lambda_nI_H + a_n$,
$\{\lambda_n\}$ being a sequence in $\mathbb{C}$ that converges to some $\lambda\in \mathbb{C}$, $\{a_n\}$ being a sequence in $K(H)$ that
converges to $\{0\}$, and $x(\infty) = \lambda I_H$. Then $((A(t))_{t\in T},\Theta)$ is a continuous field of postliminal \st -algebras. Now the
largest liminal ideal of $A(n)$ is $B(n) = K(H)$ and the largest liminal ideal of $A(\infty)$ is $B(\infty) = A(\infty) = \mathbb{C}I_H$.
Clearly $x\in \Theta$ satisfies $x(t)\in B(t)$ for every $t\in T$ only if $x(\infty) = 0$ and the continuous field is not tame.

In the next section we shall show that the continuous fields of postliminal \st -algebras in a certain class that properly includes the locally
trivial continuous fields are always tame. Another result which we prove is a necessary and sufficient condition for tameness expressed in terms
related to the minimal primitive ideals of the fibers. Afterwards we shall exhibit an example of a continuous field of postliminal \st -algebras
such that all its fibers are mutually isomorphic and its restriction to any open subset of the base space is not tame.

Let $A$ be the \st -algebra of the continuous field $((A(t)),\Theta)$ of \st -algebras over the locally compact Hausdorff space $T$ as defined
in \cite[10.4.1]{D}. By \cite[Theorem 1.1]{F} for every primitive ideal $P$ of $A$ there exist a unique $t_P\in T$ and a unique primitive ideal
$Q_P$ of $A(t_P)$ such that $P = \{x\in A \mid x(t_P)\in Q_P\}$ and conversely every pair $(t,Q)$, $t\in T,\ Q\in \Pm(A(t)$ determines a
primitive ideal of $A$ in this manner. Obviously $P$ is minimal in $\Pm(A)$ if and only if $Q_P$ is minimal in $A(t_P)$.

We shall use the terminology and the notation for continuous fields as introduced in \cite[Chapter 10]{D}. The preference to work with
continuous fields rather than Banach bundles as it is more common nowadays is motivated by the fact that Dixmier's original question was
expressed in these terms. The closed unit ball of the Banach space $X$ is denoted $X_1$.

\section{Results} \label{S:R}

The main ingredient in the proof of Proposition \ref{P:B} that follows is Michael's selection theorem \cite[Theorem 3.2'']{M}: a multivalued map
$\varphi$ from a paracompact space $T$ to the family of the non-void closed convex subsets of a Banach space $X$ that is lower semicontinuous
admits a continuous selection, i.e., there is a continuous function $f : T\to X$ such that $f(t)\in \varphi(t)$ for every $t\in T$. Moreover, if
$F$ is a closed subset of $T$ and $g : F\to X$ is a continuous selection for $\varphi\mid_F$ then one may choose $f$ so $f\mid_F = g$ is
satisfied. Recall that $\varphi$ is called lower semicontinuous if for each open subset $U$ of $X$ the set $\{t\in T \mid \varphi(t)\cap U\neq
\mset\}$ is open.

\begin{prop} \label{P:B}

   Let $T$ be a paracompact space or a locally compact Hausdorff space and $X$ a Banach space. Denote by $\mathcal{M}$ the space of the closed
   unit balls of all the closed subspaces of $X$ endowed with the Hausdorff metric. Suppose $t\to X(t)_1$, $t\in T$, is a continuous map into
   $\mathcal{M}$, $X(t)$ being a closed subspace of $X$. With $\Gamma$ the space of all the continuous functions $\varphi : T\to X$ such that
   $\varphi(t)\in X(t)$, $t\in T$, $((X(t)),\Gamma)$ is a continuous field of Banach spaces.

\end{prop}

\begin{proof}

   The only evidence we must provide is that for $t_0\in T$ and $x_0\in X(t_0)$ there exists $\varphi\in \Gamma$ such that $\varphi(t_0) = x_0$.
   Clearly we may suppose $x_0\neq 0$. Set $y_0 := x_0/\|x_0\|$. We claim that $t\to X(t)_1$ is lower semicontinuous as a multivalued map from $T$ to
   $X$. To see this let $U$ be an open subset of $X$, $s\in \set{t\in T\mid U\cap X(t)_1\neq \mset}$, and $z\in U\cap X(s)_1$. There are an open ball
   of $X$ of center $z$ and of radius $\varepsilon > 0$ contained in $U$ and a neighborhood $V$ of $s$ in $T$ such that $d(X(s)_1,X(t)_1) < \varepsilon$ for all
   $t\in V$, $d$ being the Hausdorff metric. Thus for each $t\in V$ there is $w_t\in X(t)_1$ for which $\|z - w_t\| < \varepsilon$. It follows
   that $V\subset \set{t\in T\mid U\cap X(t)\neq \mset}$ and we conclude that $\set {t\in T\mid U\cap X(t)\neq \mset}$ is open. We obtained that
   the map $t\to X(t)_1$ is indeed lower semicontinuous.

   Suppose now that $T$ is paracompact. By Michael's selection theorem mentioned above there exists a continuous map $\varphi^{\pr} :
   T\to X$ such that $\varphi^{\pr}(t)\in X(t)$ for every $t\in T$ and $\varphi^{\pr}(t_0) = y_0$. The map $\varphi$ defined by $\varphi(t) :=
   \|x_0\|\varphi^{\pr}(t)$ suits the requirements.

   Let now $T$ be locally compact Hausdorff. Let $W$ be a compact neighborhood of $t_0$. Again by Michael's selection theorem there is
   a continuous map $\varphi^{\pr} : W\to X$ such that $\varphi^{\pr}(t)\in X(t)$ for every $t\in W$ and $\varphi^{\pr}(t_0) = y_0$. Let
   now $f : T\to [0,1]$ be a continuous function such that $f(t_0) = 1$ and $f(t) = 0$ for $t\notin \mathit{Int}(W)$. The function $\varphi :
   T\to X$ defined by
   \[
      \varphi(t) :=
      \begin{cases}
         \|x_0\|f(t)\varphi^{\pr}(t),  &\text{if $t\in W$},\\
         0,                     &\text{if $t\notin W$}.
      \end{cases}
   \]
   is continuous, satisfies $\varphi(t)\in X(t)$ for $t\in T$ and $\varphi(t_0) = x_0$.

\end{proof}

A continuous field of Banach spaces over a paracompact or a locally compact Hausdorff space $T$ isomorphic to a continuous field of Banach
spaces as described in Proposition~\ref{P:B} will be called uniform. Obviously, a trivial continuous field of Banach spaces is uniform. A
continuous field of Banach spaces over $T$ is called locally uniform if $T$ has an open cover $\set{U_{\alpha}}$ such that its restriction to
each $U_{\alpha}$ is uniform. By using the regularity of the base space as was done in the last paragraph of the previous proof one gets the
following Proposition.

\begin{prop} \label{P:B1}

   Let $T$ be as in Proposition \ref{P:B}, $\{U_{\alpha}\}_{\alpha\in \A}$ an open cover of $T$ and $X_{\alpha}$ a Banach space, $\alpha\in \A$.
   Denote by $M_{\alpha}$ the space of the closed unit balls of all the closed subspaces of $X_{\alpha}$ endowed with the Hausdorff metric.
   Suppose $X(t)$, $t\in T$, is a Banach space that is a closed subspace of $X_{\alpha}$ whenever $t\in U_{\alpha}$. Moreover, suppose that the
   map $t\to X(t)_1$ from $U_{\alpha}$ into $M_{\alpha}$, $\alpha\in \A$, is continuous. Denote by $\Gamma$ the space of all the functions
   $\varphi :T\to \cup_{\alpha\in \A} X_{\alpha}$ such that $\varphi(t)\in X(t)$, $t\in T$, and restriction of $\varphi$ to $U_{\alpha}$ is continuous as a
   map into $X_{\alpha}$, $\alpha\in \A$. Then $((X(t)),\Gamma)$ is a locally uniform continuous field of Banach spaces.

\end{prop}

Of interest for us are the continuous fields of \st -algebras; of course, for uniform continuous fields of \st -algebras we shall require that
the Banach space appearing in the definition be a \st -algebra and the fibers to be \st -algebras of it. It is natural to ask if a uniform
continuous field of \st -algebras has to be locally trivial. There is some indication in \cite[Theorem 4.3]{C2} that this may be the case when
the fibers are nuclear and separable. On the other hand, \cite[Theorem 3.3]{C1} provides an example of a uniform continuous field of
(non-separable) nuclear \st -algebras that is not locally trivial.

\begin{thm} \label{T:T}

   Suppose $((A(t)),\Theta)$ is a locally uniform continuous field of postliminal \st -algebras over a paracompact or locally compact
   Hausdorff space $T$. Let $B(t)$ be the largest liminal ideal of $A(t)$ and $\Theta^{\pr} := \set{x\in \Theta \mid x(t)\in B(t), t\in T}$.
   Then $((B(t)), \Theta^{\pr})$ is a locally uniform continuous field of \st -algebras over $T$.

\end{thm}

\begin{proof}

   Let $U$ be an open subset of $T$ over which $((A(t)),\Theta)$ is
   uniform.
   There is no loss of generality if we suppose that over $U$ all the fibers $A(t)$ are \st -subalgebras of a certain \st -algebra and $t\to
   A(t) _1$ is continuous for the Hausdorff metric. If $t^{\pr},t^{\pr \pr}\in U$ satisfy $d(A(t^{\pr})_1,A(t^{\pr \pr})_1) < s (\leq 1/21)$ then
   it follows from \cite[Lemma 1.10]{PR} and \cite[Theorem 2.7]{P} that $d(B(t^{\pr})_1,B(t^{\pr \pr})_1) < 7s$. Hence $t\to B(t)_1$ is
   continuous on $U$ and the conclusion follows from Proposition \ref{P:B1}.

\end{proof}

We shall discuss now the behaviour of locally uniform continuous fields of postliminal $C^*$-algebras with respect to two kinds of ideals that
give rise to canonical composition series. Recall that one says that a point $\pi$ in the spectrum $\hat{A}$ of a \st-algebra $A$ satisfies the
Fell condition if there exist a neighbourhood $V$ of $\pi$ in $\hat{A}$ and $a\in A^+$ such that $\varrho(a)$ is a projection of rank 1 for
every $\varrho\in V$. A Fell \st-algebra is a \st-algebra for which all the points in its spectrum satisfy the Fell condition, see \cite{AS} and
\cite[6.1]{Pe} where these algebras were called of Type $I_0$. Every non trivial postliminal \st-algebra has a non trivial largest Fell ideal by
\cite [Proposition 6.1.7]{Pe}. A \st-algebra $A$ is called uniformly liminal if its ideal of all the elements $a\in A$ for which the function
$\pi\to \pi(a)$ is bounded on $\hat{A}$ is dense in $A$, see \cite[p. 443]{ASS} and the references given there. Every non trivial postliminal
\st-algebra has a non trivial largest uniformly liminal ideal by \cite[Theorem 2.6 and Theorem 2.8]{ASS}.

\begin{thm}

   Let $((A(t)),\Theta)$ be a locally uniform continuous field of postliminal \st -algebras over a paracompact or locally compact Hausdorff space
   $T$. Let $B(t)$ be the largest Fell ideal of $A(t)$ and $\Theta^{\pr} := \{x\in \Theta \mid x(t)\in B(t), t\in T\}$. Then
   $((B(t)),\Theta^{\pr})$ is a locally uniform continuous field of \st -algebras.

\end{thm}

\begin{proof}

   Let $U$ be an open subset of $T$ such that the restriction of $((A(t)),\Theta)$ to $U$ is uniform; to simplify matters we shall suppose
   that all the \st -algebras $A(t)$, \\$t\in U$ are \st subalgebras of a certain \st algebra and the map $t\to A(t)_1$ is continuous on $U$ for the
   Hausdorff metric. Let now $t_1,t_2\in U$ satisfy $d(A(t_1)_1,A(t_2)_1) < s (<1/147)$. By \cite[Lemma 1.10]{PR} there is an ideal $J$ of
   $A(t_2)$ such that $d(B(t_1)_1,J_1) < 7s$. The same lemma yields a homeomorphism h from $\hat{J}$ onto $\widehat{B(t_1)}$. Pick $\pi_0\in
   \hat{J}$ and set $\varrho_0 = h(\pi_0)\in \widehat{B(t_1)}$. There is $a\in B(t_1)^+_1$ and a neighbourhood $V$ of $\varrho_0$ in $\widehat{B(t_1)}$ such that
   $\varrho(a)$ is a rank 1 projection for $\varrho\in V$. The proof of \cite[Lemma 2.4]{PR} yields an element $b\in B(t_2)^+_1$ such that
   $\pi(b)$ is a rank 1 projection when $\pi\in h^{-1}(V)$. It follows that $J$ is a Fell ideal contained in $B(t_2)$. Similarly, there is a
   Fell  ideal $L$ of $A(t_1)$ such that $d(L_1,B(t_2)_1) < 7s$ which has to satisfy $B(t_1)\subset L$ by \cite[Lemma 1.10]{PR}. Thus $L =
   B(t_1)$ and $d(B(t_1)_1B(t_2)_1) < 7s$. We found that $t\to B(t)_1$ is continuous on $U$ and Proposition \ref{P:B1} provides the conclusion.

\end{proof}

Very likely the hypothesis of separability in the next result is unnecessary but we were not able to find a proof that dispenses with it.

\begin{thm}

   Let $((A(t)),\Theta)$ be a locally uniform continuous field of separable postliminal \st -algebras over a paracompact or locally compact
   Hausdorff space $T$. Let $B(t)$ be the largest uniformly liminal ideal of $A(t)$ and $\Theta^{\pr} := \{x\in \Theta \mid x(t)\in B(t),
   t\in T\}$. Then $((B(t)),\Theta^{\pr})$ is a locally uniform continuous field of \st -algebras.

\end{thm}

\begin{proof}

   As in the previous proof, we consider an open subset $U$ of $T$ such that the given continuous field of \st -algebras restricted to $U$ is
   uniform. We shall also suppose that the family of \st -algebras $\{A(t) \mid t\in U\}$ is contained in a certain \st -algebra and the map
   $t\to A(t)_1$ is continuous on $U$ for the Hausdorff metric. If $t_1,t_2\in U$ satisfy $d(A(t_1)_1,A(t_2)_1) < s < 1/420000$ then $A(t_1)$
   and $A(t_2)$ are isomorphic by \cite[Theorem 4.3]{C2}. Of course, every isomorphism between these two \st -algebras maps $B(t_1)$ onto
   $B(t_2)$. Hence, the same result of \cite{C2} tells us that $d(B(t_1),B(t_2)) < 28s^{1/2}$ so the map $t\to B(t)_1$ is continuous on $U$.
   Once more Proposition \ref{P:B1} yields the conclusion.

\end{proof}

\begin{qu}

   Let $((A(t)),\Theta)$ be a uniform continuous field of postliminal \st -algebras over $T$. Can one choose a non-trivial continuous trace ideal
   $B(t)$, $t\in T$,
   of $A(t)$ such that with $\Theta^{\pr} := \set{x\in \Theta \mid x(t)\in B(t), t\in T}$, $((B(t)),\Theta^{\pr})$ is a continuous field of \st
   -algebras?

\end{qu}

Given a continuous field $((A(t)),\Theta)$ of postliminal \st -algebras over a locally compact Hausdorff space $T$, let $A$ be the \st -algebra
defined by this continuous field; obviously it is a postliminal \st -algebra. Let $B$ be its greatest liminal ideal. If the image of $B$ in
$A(t)$ by the evaluation map is the greatest liminal ideal of $A(t)$, $t\in T$, then it is easily seen that the given field is tame. Conversely,
suppose that $((A(t)),\Theta)$ is tame and let $C$ be the \st -algebra defined by the continuous field of the greatest liminal ideals. Then $C =
B$ the greatest liminal ideal of $A$. Indeed, it is clear that $C$ is a liminal ideal of $A$ so $C\subset B$. Let now $x\in B$. With $t\in T$,
$\rho\in \widehat{A(t)}$, we have that $y\to \rho(y(t))$, $y\in A$, is an irreducible representation of $A$ hence $\rho(x(t))$ is a compact
operator over the space of the representation. We conclude by \cite[4.2.6]{D} that $x(t)\in C(t)$. Thus $B\subset C$.

Very likely the following lemma is known but in the absence of a reference for it we give its simple proof.

\begin{lem} \label{L:min}

   Let $A$ be a postliminal \st -algebra, $M\subset \Pm(A)$ the set of all the minimal primitive ideals. Then $\Int(M)$ is the primitive ideal
   space of the greatest liminal ideal $I$ of $A$.

\end{lem}

\begin{proof}

   Dixmier remarked in \cite[Remarque C, p. 111]{D1} that we have $\Pm(I)\subset \Int(M)$. $\Int(M)$ is the primitive ideal space of an ideal,
   $J $ say, of $A$. If $P\in \Int(M)$ then the relative closure of $\{P\}$ in $\Int(M)$ is just $\{P\}$ since all the ideals in $\Int(M)$ are
   minimal primitive ideals. Thus $\Int(M)$ is $T_1$ in the relative topology and we gather that $J$ is a liminal ideal of $A$. Hence $J\subset
   I$ and $\Int(M)\subset \Pm(I)$

\end{proof}

The set $M$ of above need not be open; for examples see \cite{G} and \cite[Example 4.3]{LT}.

We can now state and prove a necessary and sufficient condition for the tameness of a continuous field of postliminal \st -algebras.

\begin{thm}

   Let $\mathcal{C} := ((A(t)),\Theta)$ be a continuous field of postliminal \st -algebras over the locally compact Hausdorff space $T$ and $A$
   the \st -algebra defined by $\mathcal{C}$. Denote by $W$ the set of non minimal ideals in $\Pm(A)$, $W(t) := W\cap \Pm((A(t))$,
   $I(t) := \cap \{P \mid P\in W(t)\}$ for $t\in T$, and $\Tilde{\Theta} := \{x\in \Theta \mid x(t)\in I(t), t\in T\}$. Then $\mathcal{C}$ is tame if and
   only if $((I(t)),\Tilde{\Theta})$ is a continuous field of \st -algebras.

\end{thm}

\begin{proof}

   By Lemma \ref{L:min} and the remarks preceding it $\mathcal{C}$ is tame if and only if $\Pm(A(t))\cap \overline{W} = \overline{W(t)}$ for every
   $ t\in T$. Let $s\in T$ and set $J(s) := \{x(s) \mid x\in \Theta, x(t)\in I(t)\ \text{for all}\ t\in T\}\subset I(s)$. Then it follows from \cite[Theorem 1.2]{F}
   that $\Pm(A(s))\cap \overline{W} = \mathrm{hull}(J(s))$. Now $\overline{W(t)} = \mathrm{hull}(I(t))$ so $\mathcal{C}$ is tame if and only if
   $\mathrm{hull}(J(t) = \mathrm{hull}(I(t))$ for each $t\in T$ if and only if $\overline{J(t)} = I(t)$ for each $t\in T$. This means that
   $\mathcal{C}$ is tame if and only if $((I(t)),\tilde{\Theta})$ is a continuous field of \st -algebras and we are done.

\end{proof}

\section{An Example}

As mentioned in the Introduction, we are going to construct in this section a continuous field of postliminal \st -algebras over $[0,1]$ whose
fibers are mutually isomorphic and which has the additional property that none of its restrictions to the relatively open subsets of $[0,1]$ is
tame.

First we proceed to prepare two \st -algebras that will serve as building blocks of the fibers. Let $\mathbb{N} = \cup_{p=1}^{\infty} S_p$ where
the sets $\{S_p\}$ are mutually disjoint and each $S_p = \{n_1^p < n_2^p < \ldots \}$ is infinite. Let $H$ be a separable Hilbert space with an
orthonormal basis $\{\xi_k\}_{k=1}^{\infty}$ and $\mathcal{B}(H)$ the \st -algebra of all the bounded operators on $H$. Denote by $e_{ij}^0$ the
partial isometry that maps $\xi_j$ to $\xi_i$ and vanishes on each $\xi_k$ with $k\neq j$. The \st -subalgebra of \BH \ generated by
$\set{e_{ij}^0 \mid i,j = 1,2,\ldots }$ is the ideal of all compact operators and we shall denote it by $A_0$. Put now $e_{ij}^1 :=
\sum_{m=1}^{\infty} e_{n_m^in_m^j}^0$, $i,j = 1,2,\ldots $ where the series converges in the strong operator topology. Then
$\{e_{ii}^1\}_{i=1}^{\infty}$ are mutually orthogonal projections, $\sum_{i=1}^{\infty} e_{ii}^1 = \mathbf{1}_H$, $e_{ij}^1$ is a partial
isometry from $e_{jj}^1(H)$ onto $e_{ii}^1(H)$, $(e_{ij}^1)^* = e_{ji}^1$ and
\[
 e_{ij}^1e_{rs}^1 =
 \begin{cases}
    e_{is}^1,    &j = r\\
    0,           &j\neq r.
 \end{cases}
\]
Hence the \st -subalgebra $K_1$ of \BH generated by $\{e_{ij}^1 \mid i,j = 1,2,\ldots \}$ is isomorphic to $A_0$ and $A_0\cap K_1 = \{0\}$. We
have
 \begin{equation}  \label{E:1}
    e_{ij}^1e_{rs}^0 =
    \begin{cases}
       e_{n_m^is}^0,  &\text{if $r = n_m^j$ for some $m$},\\
       0,             &\text{otherwise}
    \end{cases}
\end{equation}
and
\begin{equation}  \label{E:2}
   e_{rs}^0e_{ij}^1 =
   \begin{cases}
      e_{rn_m^j}^0,  &\text{if $s = n_m^i$ for some $m$},\\
      0,             &\text{otherwise}.
 \end{cases}
\end{equation}

$A_1 := A_0 + K_1$ is a postliminal \st -algebra since $A_0$ and $K_1 \sim A_1/A_0$ are postliminal \st -algebras (actually liminal \st
-algebras in this case). Each $x\in A_1$ admits a unique decomposition $x = x_0 + x_{K_1}$ with $x_0\in A_0$ and $x_{K_1}\in K_1$. The map $x\to
x_{K_1}$ is a homomorphism hence $\|x_{K_1}\|\leq \|x\|$ and $\|x_0\|\leq 2\|x\|$.

From (\ref{E:1}) and (\ref{E:2}) it follows that the sequence $\{\sum_{i=1}^m e_{ii}^1\}_{m=1}^{\infty}$ is an increasing approximate unit for
$A_1$ consisting of projections.

We suppose now that the \st -subalgebras $K_1, \ldots K_{l-1}$ of \BH \ have been defined such that $K_p$, $1\leq p\leq l-1$, is spanned by
\[
 e_{ij}^p := \sum_{m=1}^{\infty} e_{n_m^in_m^j}^{p-1},   i,j = 1,2,\ldots .
\]
It follows that $\{e_{ii}^p\}_{i=1}^{\infty}$ are mutually orthogonal projections, $\sum_{i=1}^{\infty} e_{ii}^p = \mathbf{1}_H$, and $e_{ij}^p$
is a partial isometry from $e_{jj}^p(H)$ onto $e_{ii}^p(H)$. With $A_p := A_{p-1} + K_p$, $1\leq p\leq l-1$, we have $A_{p-1}\cap K_p = \{0\}$,
$A_p$ is a postliminal \st -subalgebra of \BH, $A_{p-1}$ is an ideal of $A_p$ and $\{\sum_{i=1}^m e_{ii}^p\}_{m=1}^{\infty}$ is an increasing
approximate unit of $A_p$ consisting of projections. Every element $x\in A_p$ admits a unique decomposition $x = x_{p-1} + x_{K_p}$ where
$x_{p-1}\in A_{p-1}$, $x_{K_p}\in K_p$. Moreover, $x\to x_{K_p}$ is a homomorphism hence $\|x_{K_p}\|\leq \|x\|$ and $\|x_{p-1}\|\leq 2\|x\|$.

We define now
\[
 e_{ij}^l := \sum_{m=1}^{\infty} e_{n_m^in_m^j}^{l-1},   i,j = 1,2,\ldots .
\]
Then $\{e_{ii}^l\}_{i=1}^{\infty}$ are mutually orthogonal projections, $\sum_{i=1}^{\infty} e_{ii}^l = \mathbf{1}_H$, and $e_{ij}^l$ is a
partial isometry from $e_{jj}^l(H)$ onto $e_{ii}^l(H)$. We obtain $(e_{ij}^l)^* = e_{ji}^l$ and
\begin{equation} \label{E:3}
   e_{ij}^le_{rs}^l =
   \begin{cases}
      e_{is}^l,   &\text{if $j=r$},\\
      0,          &\text{otherwise},
   \end{cases}
\end{equation}
hence the \st -subalgebra $K_l$ of \BH \ generated by $\{e_{ij}^l \mid i,j = 1,2,\ldots \}$ is isomorphic to $A_0$. We also have
\begin{equation}  \label{E:4}
   e_{ij}^le_{rs}^{l-1} =
   \begin{cases}
      e_{n_m^is}^{l-1},   &\text{if $r = n_m^j$ for some $m$},\\
      0,                 &\text{otherwise},
   \end{cases}
\end{equation}
and
\begin{equation}   \label{E:5}
   e_{rs}^{l-1}e_{ij}^l = (e_{ji}^le_{sr}^{l-1})^* =
   \begin{cases}
      e_{rn_m^j}^{l-1},  &\text{if $s = n_m^i$ for some $m$},\\
      0,                 &\text{otherwise}.
   \end{cases}
\end{equation}
Hence, if $x\in K_l$ and $y\in K_{l-1}$ then $xy,yx\in K_{l-1}$. Suppose now that $x\in K_l$, $z\in A_{l-2}$. Then
\[
 xz = \lim_{m\to \infty} x(\sum_{i=1}^m e_{ii}^{l-1})z = \lim_{m\to \infty} (x\sum_{i=1}^m e_{ii}^{l-1})z.
\]
We established that $x\sum_{i=1}^m e_{ii}^{l-1}\in K_{l-1}$ for every $m$ hence $xz\in A_{l-2}$. Similarly, $zx\in A_{l-1}$ and we gather that
$A_{l-1} = A_{l-2} + K_{l-1}$, $k\geq 2$, is an ideal in $A_l := A_{l-1} + K_l$ which is a \st -subalgebra of \BH \ by \cite[1.8.4]{D}. We want
now to prove $A_{l-1}\cap K_l = \{0\}$. Denote by $B_l^m$ the finite dimensional \st -algebra generated by $\{e_{ij}^l \mid 1\leq i,j\leq m\}$.
Then $K_l = \overline{\cup_{m=1}^{\infty} B_l^m}$. From
\[
 \|\sum_{i,j=1}^m \alpha_{ij}e_{ij}^l - \sum_{i,j=1}^m \alpha_{ij}e_{ij}^l\sum_{r=1}^s e_{rr}^{l-1}\| = \|\sum_{i,j=1}^m \alpha_{ij}e_{ij}^l\|
\]
for every $s$ we get $B_l^m\cap A_{l-1} = \{0\}$ for every $m$. Thus, the quotient map $A_l\to A_l/A_{l-1}$ is isometric on each $B_l^m$ hence
it is isometric on $K_l$ and we conclude that $A_{l-1}\cap K_l = \{0\}$. $A_l$ is a postliminal \st -algebra since $A_{l-1}$ and
$A_l/A_{l-1}\sim K_i$ are postliminal \st -algebras.

In this manner we construct inductively an increasing sequence $\{A_l\}_{l=0}^{\infty}$ of postliminal \st -subalgebras of \BH \ such that
$A_{l-1}$ is an ideal in $A_l$. It follows that $A := \overline{\cup_{l=0}^{\infty} A_l}$ is a postliminal \st -subalgebra of \BH \ whose
greatest liminal ideal is $A_0$.

Set now $C_1 := K_1$, $C_2 := C_1 + K_2$. $C_1$ is an ideal in $\overline{C_2}$ hence $C_2$ is closed by \cite[1.8.4]{D}. We have $C_1\cap K_2 =
\{0\}$ and $A_0\cap C_2 = \{0\}$ since $(A_0 + K_1)\cap K_2 = \{0\}$ and $A_0\cap K_1 = \{0\}$. We define inductively $C_l : C_{l-1} + K_l$.
Then $C_{l-1}$ is an ideal of the \st -algebra $C_l$, $C_{l-1}\cap K_l = \{0\}$ and $A_0\cap C_l = \{0\}$. From (\ref{E:3}), (\ref{E:4}) and
(\ref{E:5}) we find that $e_{ij}^l\to e_{ij}^{l-1}$, $1\leq l\leq p$, $i,j\geq 1$ yields an isomorphism $\varphi_p$ of $C_p$ onto $A_{p-1}$.
Obviously $\varphi_{p+1}$ extends $\varphi_p$ hence one gets an isomorphism $\varphi$ from $C := \overline{\cup_{p=1}^{\infty} C_p}$ onto $A$
that extends each $\varphi_p$. Let now $x\in A$, $x = \lim_{p\to \infty} x_p$ with $x_p\in A_p$, $p\geq 1$. Then $x_p = x_p^0 + x_p^p$ where
$x_p^0\in A_0$ and $x_p^p\in C_p$. From $\|x_p^0 - x_q^o\|\leq 2\|x_p - x_q\|$ we conclude that the Cauchy sequence $\{x_p^0\}_{p=1}^{\infty}$
converges to some $x^0\in A_0$ hence $\{x_p^p\}_{p=1}^{\infty}$ converges to some $x^C\in C$ that satisfies $x = x^0 + x^C$. Now $x_p\to x_p^p$
is a homomorphism for each $p$ therefore $x\to x^C$ is a homomorphism. We have $\|x^C\|\leq \|x\|$ and $\|x^0\|\leq 2\|x\|$. The quotient map
$A\to A/A_0$ is isometric on each $C_p$ hence it is isometric on $C$. It follows that $A_0\cap C = \{0\}$ and the decomposition $x = x^0 + x^C$
is unique.

Now we can begin constructing the continuous field of \st -algebras we need. Let $\{r_n\}$ be an enumeration of the set of rational numbers in
$[0,1]$. For an irrational number $t\in [0,1]$ we define $A(t) := c_0(A)$ that is, the direct sum of $A$ with itself $\aleph_0$ times. $A(r_n)$
is a \st -subalgebra of $c_0(A)$ that is also a direct sum of copies of $A$ except that at the $n$-th spot we insert $C$ instead of $A$.
Obviously all the fibers are mutually isomorphic postliminal \st -algebras. The ${}^*$-algebra $\Gamma$ of the continuous vector fields consists
of all the continuous functions $x : [0,1]\to c_0(A)$ such that $x(t)\in A(t)$ for every $t\in [0,1]$.

To show that $((A(t))_{t\in [0,1]},\Gamma)$ so defined is a continuous field we must check that $\{x(t) \mid x\in \Gamma\} = A(t)$ for $t\in
[0,1]$. To this end let $t_0\in [0,1]$ and $\{a_n\}\in A(t_0)$. For $n\in \mathbb{N}$ let $f_n : [0,1]\to [0,1]$ be a continuous function such
that $f_n(t_0) = 1$ and $f_n(r_n) = 0$ if $r_n\neq t_0$. Define $x(t) := \{f_n(t)a_n\}$, $t\in [0,1]$. Then $x$ is a continuous function from
$[0,1]$ to $c_0(A)$ such that $x(t)\in A(t)$ for $t\in [0,1]$ i.e. $x\in \Gamma$. Moreover $x(t_0) = \{a_n\}$ and we have proved that we
constructed a continuous field of \st -algebras.

The greatest liminal ideal $B(t)$ of $A(t)$ is $c_0(A_0)$ when $t$ is irrational. The greatest liminal ideal $B(r_n)$ of $A(r_n)$, $n\in
\mathbb{N}$, is again a direct sum whose components are all equal to $A_0$ except the one at the $n$-th place that it is equal to $K_1$. Thus if
$x\in \Gamma$ satisfies $x(t)\in B(t)$ for every $t\in [0,1]$ then the $n$-th component of $x(r_n)$ must vanish since $A_0\cap C =\{0\}$. It
follows that for the restriction of our continuous field of \st -algebras to any relatively open subset $U$ of $[0,1]$ the family $(B(t))_{t\in
U}$ together with $\{x\in \Gamma \mid x(t)\in B(t), t\in U\}$ does not form a continuous field of \st -algebras.
\bibliographystyle{amsplain}
\bibliography{}

\end{document}